\documentclass[12 pt]{amsart}
\usepackage{amsmath,amssymb,amsthm}
\usepackage{longtable}
\usepackage{color}
\usepackage[width=7in]{geometry}

\newcommand{\PSL}{\mathrm{PSL}}

\newtheorem{defn}{Definition}[section]
\newtheorem{theorem}[defn]{Theorem}
\newtheorem{remark}[defn]{Remark}
\newtheorem{lemma}[defn]{Lemma}
\newtheorem{corollary}[defn]{Corollary}

\begin{document}

\title{Coprime subdegrees of twisted wreath permutation groups}
\author{Alexander Y. Chua, Michael Giudici and Luke Morgan}

\address{Centre for the Mathematics of Symmetry and Computation\\
School of Mathematics and Statistics\\
The University of Western Australia\\
35 Stirling Highway\\
Crawley\\
WA 6009\\
Australia} 

	\email{21506815@student.uwa.edu.au, michael.giudici@uwa.edu.au, luke.morgan@uwa.edu.au}

\begin{abstract} Dolfi, Guralnick, Praeger and Spiga asked if there exist infinitely many primitive groups of twisted wreath type with nontrivial coprime subdegrees. Here we settle this question in the affirmative. We construct infinite families of primitive twisted wreath permutation groups with nontrivial coprime subdegrees. In particular, we define a primitive twisted wreath group $G(m,q)$ constructed from the nonabelian simple group $\text{PSL}(2,q)$ and a primitive permutation group of diagonal type with socle $\PSL(2,q)^m$, and determine all values of $m$ and $q$ for which $G(m,q)$ has nontrivial coprime subdegrees. In the case where $m=2$ and $q\notin\{7,11,29\}$ we obtain a full classification of all pairs of nontrivial coprime subdegrees. \end{abstract}

\maketitle

\section{Introduction} 

If $G$ is a transitive permutation group acting on a finite set $\Omega$ and we fix some point $\alpha\in\Omega$, a \textit{subdegree} of $G$ relative to $\alpha$ is defined as the size of a $G_{\alpha}$-orbit. These are the sizes of the sets $\beta^{G_{\alpha}}$ where $\beta\in\Omega$, or equivalently, the values of $|G_{\alpha}:G_{\alpha}\cap G_{\beta}|$. The subdegree is said to be \textit{trivial} if it corresponds to the $G_{\alpha}$-orbit $\{\alpha\}$, and \textit{nontrivial} otherwise. If $G$ is primitive and not cyclic of prime order, then the only subdegree equal to $1$ is the trivial subdegree, so all nontrivial subdegrees are greater than $1$. The study of subdegrees is a classical topic in permutation group theory. Probably the most famous result is the verification of the Sims conjecture ~\cite{Sims} that bounds the order of point stabilisers in primitive groups in terms of their subdegrees. 

Primitive groups are classified into types by the O'Nan-Scott Theorem, using the subdivision in \cite{Inclusion}. The primitive groups of twisted wreath type (TW) are the most mysterious and commonly misunderstood. We refer the reader to \cite[Section 4.7]{PermDixon} and \cite{TWBaddeley} for detailed treatments and provide more information in Section \ref{TW}. This paper deals with subdegrees of twisted wreath groups, a topic that does not appear very often in the literature. The published results include a paper by Giudici, Li, Praeger, Seress and Trofimov ~\cite{GiudiciSub}, which proves bounds on the minimal subdegrees and explicitly constructs such a minimal $G_{\alpha}$-orbit. A result by Fawcett in her PhD thesis ~\cite[p. 59]{Fawcett} shows that if the point stabiliser $G_{\alpha}$ acts primitively on the set of simple direct factors of the socle, then there is a subdegree of size $|G_{\alpha}|$. 

The study of coprime subdegrees dates back to the work of Marie Weiss in 1935, who proved that if $G$ is primitive with coprime subdegrees $m$ and $n$, then $G$ has a subdegree dividing $mn$ that is greater than both $m$ and $n$. Moreover, if it has $k$ pairwise coprime subdegrees then it has rank at least $2^k$ (see \cite[p. 92-93]{SubdegreeResults}). The motivation behind this paper was a result of Dolfi, Guralnick, Praeger and Spiga ~\cite{CopSubPraeger} that is proven in \cite{proof}, stating that the maximal size of a set of pairwise coprime nontrivial subdegrees of a finite primitive permutation group is at most $2$. Dolfi et al. also showed that if a primitive permutation group has a pair of nontrivial coprime subdegrees, then its type is Almost Simple (AS), Product Action (PA) or Twisted Wreath (TW). For types AS and PA they constructed infinite families with nontrivial coprime subdegrees, but with type TW only one example is known. 

In this paper we construct a primitive TW group $G(m,q)$ determined from the nonabelian simple group $\text{PSL}(2,q)$ and a primitive permutation group of diagonal type with socle $\PSL(2,q)^m$. The group $G(2,7)$ is the example given in \cite{CopSubPraeger} of a primitive TW group with nontrivial coprime subdegrees. For the full definition of $G(m,q)$, see Remark \ref{define GMQ}. In Table \ref{list} we calculate a number of different subegrees of $G(m,q)$, from which we find infinitely many values of $m$ and $q$ for which $G(m,q)$ has nontrivial coprime subdegrees, as presented in Table \ref{coprime subdegrees}. We go further with our analysis and in Sections \ref{any m at least 2} and \ref{m=2} we show that these are the only such pairs. This gives the following main result:

\begin{theorem} The group $G(m,q)$ has nontrivial coprime subdegrees if and only if one of the following holds:
\begin{enumerate} 
\item $q\equiv 3\pmod{4}$ or $q=29$,
\item $q$ is even and $m\geqslant 3$.
\end{enumerate}
\end{theorem}

Finally, in Section \ref{m=2} we analyse the $m=2$ case in more detail to determine all pairs of nontrivial coprime subdegrees, and we manage to do this for all $q\notin \{7,11,29\}$, given in Table \ref{final list}. For these three small cases we have an inclusion statement about the two-point stabiliser, which gives some information about the subdegree. Remarkably, we find that if $q\equiv 3\pmod{4}$ and $q>19$, then $G(2,q)$ has exactly one pair of nontrivial coprime subdegrees.   

\section{Twisted wreath groups}\label{TW}

We now describe the construction of the twisted wreath product as introduced by Neumann in \cite{NeumannTW}. Let $T$ and $H$ be arbitrary groups. For any subset $X$ of $H$, let $T^X$ denote the set of functions from $X$ to $T$, which is a group under pointwise multiplication. Let $\text{id}$ denote the function defined by $f(x)=1$ for all $x\in X$. It can be shown that $H$ acts as a group of automorphisms on $T^H$ by $f^x(z)=f(xz)$ for all $f\in T^H$ and $x,z\in H$. Now let $L$ be a subgroup of $H$ and $R$ be a set of left coset representatives of $L$ in $H$. Let $\phi:L\rightarrow\text{Aut}(T)$ be a homomorphism. Set
\[ N = \{f\in T^H \mid f(z\ell)=f(z)^{\phi(\ell)} \text{ for all } z\in H \text{ and } \ell\in L\}.\]
We can show that $N$ is a subgroup of $T^H$ and that $N\cong T^R$. Furthermore, the group $N$ is invariant under the action of $H$, so $H$ acts as a group of automorphisms on $N$. 

\begin{defn}\label{TWconstruction} We define the twisted wreath product determined from $(T,H,\phi)$ to be the group $G=N\rtimes H$. The group $G$ acts on $\Omega=N$ with $N$ acting by right multiplication and $H$ acting by automorphisms, that is, $\alpha^{nh}=(\alpha n)^h$ for all $\alpha\in\Omega, n\in N$ and $h\in H$. \end{defn}

\begin{lemma}\label{basic} The nontrivial subdegrees of $G$ are the values of $|H:H_f|$ for $f\in N\backslash\{\text{id}\}$. Also, no nontrivial subdegrees of $G$ are equal to $1$. \end{lemma}

\begin{proof} We can verify that $G_{\text{id}}=H$, so the nontrivial subdegrees of $G$ are of the form $|G_{\text{id}}:G_{\text{id}}\cap G_f|=|H:H_f|$ for some $f\in N\backslash\{\text{id}\}$. Since $G$ is not cyclic of prime order, no nontrivial subdegrees of $G$ are equal to $1$. \end{proof} 

The following result from \cite[Lemma 4.7A]{PermDixon} gives a set of sufficient conditions for a twisted wreath product to be primitive. 

\begin{theorem}\label{TWsufficient} Let $T$ be a finite nonabelian simple group, and suppose $H$ is a primitive permutation group with point stabiliser $L$. Suppose that the group of inner automorphisms of $T$ is contained in the image of $\phi$, but $\text{Im }\phi$ is not a homomorphic image of $H$. Then the twisted wreath product determined from $(T,H,\phi)$ is a primitive group with regular socle $N$, and $N\cong T^m$ where $m = |H:L|$. \end{theorem}

We will deal with a class of primitive TW groups constructed from a group of diagonal type. 

\begin{lemma}\label{will be used later} Let $T$ be a finite nonabelian simple group, let $H=T\wr S_m$ and let $L=\{(x,\ldots, x)\sigma\mid x\in T,\sigma\in S_m\}$. Define $\phi:L\rightarrow \text{Aut}(T)$ by setting $\phi((x, \ldots, x)\sigma)=i_x$ for all $x\in T$ and $\sigma\in S_m$, where $i_x$ denotes the automorphism of $T$ induced by conjugation by $x$. Then the construction in Definition \ref{TWconstruction} yields a primitive TW permutation group $G(m,T)$ with socle isomorphic to $T^{|T|^{m-1}}$ and point stabiliser isomorphic to $T\wr S_m$. \end{lemma}

\begin{proof} The group $H$ acts primitively on the set of right cosets of $L$. Note that $\text{Inn}(T)=\text{Im }\phi$ and that $\text{Inn}(T)$ is not a homomorphic image of $H$. All the conditions of Theorem \ref{TWsufficient} have been satisfied, so $G(m,T)$ is indeed a primitive group of type TW. Since $|H:L|=\frac{|T|^m\cdot m!}{|T|\cdot m!}=|T|^{m-1}$, the socle is of the form $T^{|T|^{m-1}}$. \end{proof}

\begin{remark}\label{define GMQ} We define $G(m,q)=G(m,\PSL(2,q))$ and note that $G(2,7)$ is the primitive TW group in \cite[p. 12--14]{CopSubPraeger} with nontrivial coprime subdegrees. \end{remark} 

Throughout, we will let $H,L$ and $\phi$ be as defined in Lemma \ref{will be used later}. Let 
\[ N=\{f\in T^H \mid f(z\ell)=f(z)^{\phi(\ell)} \text{ for all } z\in H \text{ and } \ell\in L\}\]
be the set of functions that $G(m,T)$ acts on. 

We now construct some $g\in N$ which is very similar to the function used by Dolfi et al.~\cite{CopSubPraeger} and Giudici et al.~\cite{GiudiciSub}.

\begin{lemma}\label{fn} Let $D$ be a subgroup of $H$. Suppose there exists $t\in H$ such that $(\eta, \ldots, \eta)\sigma\in Z(D^t \cap L)$ with $\eta\neq 1$. Then $g\in T^H$ defined by
\begin{equation}\label{weird fn} g(z) = \begin{cases} \eta^{\phi(\ell)} & \text{if } z=dt\ell, \text{ for some } d\in D \text{ and } \ell\in L, \\ 1 & \text{if } z\in H\backslash DtL \end{cases} \end{equation}
is well-defined. Moreover, $g$ is a nonconstant function, $g\in N$ and $D\leqslant\textbf{C}_H(g)$. \end{lemma}

\begin{proof} Firstly, we show that $g$ is well-defined. If $z=d_1t\ell_1=d_2t\ell_2$ for $d_1,d_2\in D$ and $\ell_1,\ell_2\in L$, then $\ell_2\ell_1^{-1}=t^{-1}d_2^{-1}d_1t\in D^t\cap L$. Hence $\ell_2=u\ell_1$ with $u\in D^t\cap L$. Let $u=(\rho, \ldots, \rho)\tau$. Since $(\eta, \ldots, \eta)\sigma\in Z(D^t\cap L)$, it follows that $(\rho\eta, \ldots, \rho\eta)\tau\sigma=(\eta\rho, \ldots, \eta\rho)\sigma\tau$. In particular, $\rho\eta=\eta\rho$ and thus $\eta^{\phi(u)}=\eta^{i_{\rho}}=\eta$. Hence
\[ \eta^{\phi(\ell_1)}=\eta^{\phi(u)\phi(\ell_1)}=\eta^{\phi(u\ell_1)}=\eta^{\phi(\ell_2)} \]
and so $g(z)$ does not depend on the representation $z=d_it\ell_i$ of $z$. 

Next, we show that $g\in N$. If $z=dt\ell$ for $d\in D$ and $\ell\in L$, then 
\[ g(z\ell_1)=g(dt\ell\ell_1)=\eta^{\phi(\ell\ell_1)}=\eta^{\phi(\ell)\phi(\ell_1)}=g(z)^{\phi(\ell_1)} \]
for each $\ell_1\in L$. If $z\notin DtL$, then $z\ell \notin DtL$, so $g(z\ell)=1=g(z)^{\phi(\ell)}$. Hence $g\in N$. Now suppose for a contradiction that $g$ is constant. Then $\eta^{\phi(\ell)}=\eta^{\phi(1)}=\eta$ for all $\ell\in L$, and since $\text{Inn}(T)\leqslant \text{Im }\phi$ we have $\eta\in Z(T)=1$, a contradiction. Thus $g$ is nonconstant. 

Finally, we show that for each $d\in D$ and $z\in H$, we have $g^d(z)=g(dz)=g(z)$, by considering the $z\in DtL$ and $z\notin DtL$ cases. If $z=d_1t\ell_1$ for $d_1\in D$ and $\ell_1\in L$, then $g(dz)=g(dd_1t\ell_1)=\eta^{\phi(\ell_1)}=g(d_1t\ell_1)=g(z)$. If $z\notin DtL$, then it also follows that $dz\notin DtL$, so $g(dz)=g(z)=1$. Thus, $D$ centralises $g$. \end{proof}

\begin{lemma}\label{|H:D|} Let $D$ be a maximal subgroup of $H$. If there exists $t\in H$ such that $Z(D^t\cap L)$ contains an element $(\eta, \ldots, \eta)\sigma$ with $\eta\neq 1$, then $|H:D|$ is a subdegree of $G(m,T)$. \end{lemma}

\begin{proof} Define $g$ as in Equation \eqref{weird fn}. Then by Lemma \ref{fn}, $D\leqslant \textbf{C}_H(g)$. Since $D$ is maximal in $H$, we conclude that $\textbf{C}_H(g)=D$ or $H$. If $\textbf{C}_H(g)=H$, then for each $h\in H$, we have that $g(hz)=g^h(z)=g(z)$, and thus $g$ is a constant function, contradicting $\eta\neq 1$ and Lemma \ref{fn}. So $\textbf{C}_H(g)=D$. Hence $|H:\textbf{C}_H(g)|=|H:D|$ is a subdegree of $G(m,T)$. \end{proof}

\begin{corollary}\label{max} Let $K$ be a maximal subgroup of $T$ and let $D=K\wr S_m$. If there exists $t\in H$ such that $Z(D^t\cap L)$ contains an element $(\eta,\ldots,\eta)\sigma$ with $\eta\neq 1$, then $|T:K|^m$ is a subdegree of $G(m,T)$. \end{corollary}

\begin{proof} By Corollary 1.5A and Lemma 2.7A in \cite{PermDixon}, it follows that $D$ is maximal in $H$. The result now follows from applying Lemma \ref{|H:D|}. \end{proof}

\begin{corollary}\label{t=1} Let $K$ be maximal in $T$ with $Z(K)\neq 1$. Then $|T:K|^m$ is a subdegree of $G(m,T)$. \end{corollary}

\begin{proof} Let $D=K\wr S_m$. Then $D\cap L\cong K\times S_m$, so $Z(D\cap L)\cong Z(K)\times Z(S_m)$. Since $Z(K)\neq 1$ we can apply Corollary \ref{max} with $t=1$, and thus $|T:K|^m$ is a subdegree of $G(m,T)$. \end{proof}

\begin{corollary}\label{int with conj has centre} Suppose $m\geqslant 3$. Let $K$ be a maximal subgroup of $T$ such that there exists $s\in T\backslash K$ with $Z(K\cap K^s)\neq 1$. Then $|T:K|^m$ is a subdegree of $G(m,T)$. \end{corollary}

\begin{proof} Let $D=K\wr S_m$ and set $t=(1,\ldots,1,s)$. Then
\begin{align*} D^t\cap L&=\{(1,\ldots,1,s^{-1})(k_1,\ldots,k_m)\sigma(1,\ldots,1,s)\mid k_i\in K,\sigma\in S_m\}\cap L\\
&=\{(1,\ldots,1,s^{-1})(k_1,\ldots,k_m)\underbrace{(1,\ldots,s,\ldots,1)}_{\text{$s$ in the $m^{\sigma^{-1}}$th component}}\sigma\mid k_i\in K,\sigma\in S_m\}\cap L.\end{align*}
If $m^{\sigma^{-1}}\neq m$ then since $m\geqslant 3$ one component of the product will be of the form $k_is$, while another will be of the form $k_j$. But $k_is\neq k_j$ since $s\notin K$, so the product is not in $L$. So
\begin{align*} D^t\cap L&=\{(1,\ldots,1,s^{-1})(k_1,\ldots,k_m)(1,\ldots,1,s)\sigma\mid k_i\in K, m^{\sigma}=m\}\cap L\\
&=\{(k_1, \ldots, k_{m-1}, k_m^s)\sigma\mid k_i\in K, m^{\sigma}=m\}\cap L\\
&\cong (K\cap K^s)\times S_{m-1}. \end{align*}
Hence
\[ Z(D^t\cap L)\cong Z(K\cap K^s)\times Z(S_{m-1}) \]
and since $Z(K\cap K^s)\neq 1$, we can apply Corollary \ref{max} to conclude that $|T:K|^m$ is a subdegree of $G(m,T)$. \end{proof} 

\begin{corollary}\label{C_2 int} Let $K$ be maximal in $T$ and assume that there exists $s\in T\backslash K$ such that $K\cap K^s\cong C_2$. Then $|T:K|^m$ is a subdegree of $G(m,T)$. \end{corollary}

\begin{proof} If $m\geqslant 3$, we are immediately done by Corollary \ref{int with conj has centre} as $K\cap K^s$ is abelian. So suppose $m=2$ and let $D=K\wr S_2$. Let $M=\{(x,x)\mid x\in T\}$ and set $t=(1,s)$. Then $D^t\cap M\cong K\cap K^s\cong C_2$. Since $M\vartriangleleft L$ it follows that $D^t\cap M\vartriangleleft D^t\cap L$ and so $D^t\cap M\leqslant Z(D^t\cap L)$. So by Corollary \ref{max}, we see that $|T:K|^2$ is a subdegree of $G(m,T)$. \end{proof}

The following lemma is inspired by the construction of the subdegree $24^2$ in the group $G(2,7)$~\cite{CopSubPraeger}.  

\begin{lemma}\label{centraliser} Let $\gamma$ be a nontrivial element of $T$. Then $\left (\frac{|T|}{|\textbf{C}_T(\gamma)|}\right )^m=|\gamma^T|^m$ is a subdegree of $G(m,T)$.\end{lemma}

\begin{proof} Let $D=\textbf{C}_T(\gamma)\wr S_m$. Define $h\in T^H$ by 
\begin{equation} h(z) = \begin{cases} \gamma^{\phi(\ell)} & \text{if } z=d\ell, \text{ for some } d\in D \text{ and } \ell\in L, \\ 1 & \text{if } z\in H\backslash DL. \end{cases} \end{equation}
That is, $h$ is defined as in Equation \eqref{weird fn} with $t=1$ and $(\gamma,\ldots, \gamma)$ being a nontrivial element in $Z(D\cap L)$. Thus Lemma $\ref{fn}$ implies that $h$ is well-defined, $h\in N$ and $D\leqslant\textbf{C}_H(h)$. We cannot use maximality as before to conclude that $D=\textbf{C}_H(h)$, but we can prove this another way. Let $h_1\in \textbf{C}_H(h)$ and suppose that $h_1\notin DL$. Then $\gamma=h(1)=h^{h_1}(1)=h(h_1)=1$, a contradiction. Thus $h_1=d\ell$ for some $d\in D$ and $\ell\in L$. Since $D\leqslant\textbf{C}_H(h)$, we have $\ell=d^{-1}h_1\in \textbf{C}_H(h)$. Let $\phi(\ell)=i_x$, that is, $\ell=(x, \ldots, x)\sigma$ for some $x\in T$ and $\sigma\in S_m$. Then $\gamma=h(1)=h^{\ell}(1)=h(\ell)=\gamma^{i_x}$, so $x\in \textbf{C}_T(\gamma)$. Hence $\ell\in D$ and $h_1\in D$, so $D=\textbf{C}_H(h)$. Thus $|H:\textbf{C}_H(h)|=|H:D|=\left (\frac{|T|}{|\textbf{C}_T(\gamma)|}\right )^m=|\gamma^T|^m$ is a subdegree of $G(m,T)$. \end{proof}

We now explain how the above results could be used to construct infinite families of primitive TW groups with nontrivial coprime subdegrees. If $G$ is a finite group with subgroups $A$ and $B$, we say $G=AB$ is a \textit{coprime factorisation} if $|G:A|$ and $|G:B|$ are coprime. If $A$ and $B$ are maximal in $G$, we say $G=AB$ is a \textit{maximal coprime factorisation}. 

\begin{theorem}\label{has centre} Let $T=AB$ be a maximal coprime factorisation of the finite nonabelian simple group $T$, and suppose $m\geqslant 3$. If there exist $s,t\in T$ such that $Z(A\cap A^s)\neq 1$ and $Z(B\cap B^t)\neq 1$, then the group $G(m,T)$ has the nontrivial coprime subdegrees $|T:A|^m$ and $|T:B|^m$. \end{theorem}

\begin{proof} This follows from Corollary \ref{t=1} when $s$ or $t$ lie in $K$ and Corollary \ref{int with conj has centre} when $s$ or $t$ lie in $T\backslash K$. \end{proof}

\begin{remark} Theorem \ref{has centre} is a very powerful result as \cite{CopSubPraeger} contains a list of all the maximal coprime factorisations of finite nonabelian simple groups. \end{remark}

\begin{remark} Let $T$ be a finite nonabelian simple group. If $T=AB$ is a maximal coprime factorisation such that both $A$ and $B$ have nontrivial centres, then we will have an infinite family of primitive TW groups with nontrivial coprime subdegrees (this is a special case of Theorem \ref{has centre} with $s=t=1$). It is a conjecture of Szep that $T$ can never be written as $AB$ for any (not necessarily maximal) subgroups $A$ and $B$ of $T$ with nontrivial centre. This conjecture was proven in \cite{Szep}, meaning that this idea cannot be used to construct an infinite family. \end{remark}

\begin{remark} By Lemma \ref{centraliser}, if there exist two nontrivial conjugacy classes of coprime size, then we will have an infinite family of primitive TW groups with nontrivial coprime subdegrees. However, as observed in \cite{Szep}, this is not possible and is an immediate corollary of the Szep Conjecture. \end{remark}

\section{Results about $\text{PSL}(2,q)$}

We begin with the following standard lemma.  

\begin{lemma}\label{double counting} Let $K$ be a maximal subgroup of a simple group $T$ and let $R$ be a subgroup of $K$. Let $x$ be the number of conjugates of $R$ in $T$ that are contained in $K$ and let $y$ be the number of conjugates of $K$ in $T$ whose intersection with $K$ contains $R$. Then
\[ y=\frac{x\cdot |N_T(R)|}{|K|}. \] \end{lemma}

\begin{proof} We will count the number of pairs $(X,Y)$ of subgroups of $T$, with $X$ conjugate to $R$, with $Y$ conjugate to $K$, and $X\leqslant Y$. By fixing $X$ and considering the possibilites for $Y$, and then by fixing $Y$ and considering the possibilities for $X$, we obtain
\[ (\#\text{conjugates of } R \text{ in } T)(\#\text{conjugates of } K \text{ in } T \text{ containing } R)\]
\[ =(\#\text{conjugates of } R \text{ in } T \text{ that are contained in } K)(\#\text{conjugates of } K \text{ in } T). \]
Hence
\[ \frac{|T|}{|N_T(R)|}\cdot y = x \cdot \frac{|T|}{|N_T(K)|}. \]
Now $K\leqslant N_T(K)\leqslant T$, so by the maximality of $K$ it follows that $N_T(K)=K$ or $T$. However, the simplicity of $T$ implies that $K$ is not normal in $T$, so $N_T(K)=K$. Thus
\[ y=\frac{x\cdot |N_T(R)|}{|K|}. \qedhere \] \end{proof} 

\begin{corollary}\label{one conjugacy class} We use the same notation as Lemma \ref{double counting}. Also, suppose that there is only one conjugacy class of subgroups isomorphic to $R$ in $K$. Then
\[ y=\frac{|N_T(R)|}{|N_K(R)|}. \] \end{corollary}

\begin{proof} Since there is only one conjugacy class of $R$ in $K$, we have that $x=\frac{|K|}{|N_K(R)|}$ and the result follows. \end{proof}

We will be working a lot with the projective special linear groups. Information about their subgroups and maximal subgroups will prove to be useful. The list in Dickson ~\cite{PSLDickson} is the most commonly cited but contains an error about the number of conjugacy classes of dihedral groups. In particular, it states that for a divisor $d>2$ of $\frac{q\pm 1}{(2,q-1)}$, there is one conjugacy class of subgroups isomorphic to $D_{2d}$ for $d$ odd, and two conjugacy classes if $d$ is even. However, it is actually the case that there are two conjugacy classes for $d$ odd, and one conjugacy class for $d$ even. We also point the reader to \cite{PSLMoore} which contains a correct list. For the rest of this section, set $T=\PSL(2,q)$. We will denote the point stabiliser of a one-dimensional subspace by $P_1$, and begin with a lemma about involutions in cosets of $P_1$.

\begin{lemma}\label{P_1 involution} If $s\notin P_1$, then the coset $P_1s$ contains an involution. \end{lemma}

\begin{proof} Let $\langle v\rangle$ be the one-dimensional subspace of $V$ stabilised by $P_1$, and let $\langle w\rangle = \langle v\rangle ^s$. Since $s\notin P_1$, it follows that $\{v, w\}$ is a basis of $V$. Take the element $g\in \text{SL}(2,q)$ such that $v^g=w$ and $w^g=-v$. Any element of $V$ can be written as $\alpha v + \beta w$ for some $\alpha,\beta\in \text{GF}(q)$, and we can show that $(\alpha v+\beta w)^{g^2}=-\alpha v - \beta w$, so $g^2 = -I_2$. Let $h$ be the permutation of $\mathcal{P}_1(V)$ induced by $g$. Then for all $u\in V$ we have
\[ \langle u\rangle^{h^2}=\langle u^{g^2}\rangle = \langle -u \rangle = \langle u \rangle, \]
so $h^2$ fixes all the one-dimensional subspaces of $V$ and $h$ is an involution. Note also that $\langle v\rangle^h=\langle v^g\rangle=\langle w \rangle=\langle v\rangle^s$, so $h\in P_1s$. So the coset $P_1s$ contains an involution. \end{proof}

The next few lemmas deal with possible intersections of conjugate subgroups of $T$. We now introduce some notation. Let $K$ be a subgroup of $T$. For any $R\leqslant K$, let $f(R)$ denote the number of conjugates of $K$ whose intersection with $K$ contains $R$, and let $g(R)$ denote the number of conjugates of $K$ whose intersection with $K$ is equal to $R$. Note that for $R$ maximal in $K$, we have $f(R)=g(R)+1$. 

\begin{lemma}\label{C_2} Suppose that $q\equiv\pm 1 \pmod{10}$, and $q$ is a prime or $q=p^2$ for some prime $p\equiv\pm 3 \pmod{10}$. Let $K\cong A_5$ be a subgroup of $T$. If $q>11$ there exists $t\in T$ such that $K\cap K^t\cong C_2$. \end{lemma}

\begin{proof} Consider a fixed $C_2$ in $K$. It is contained in two copies of $D_{10}$, two copies of $D_6$ and one copy of $C_2^2$ in $K$. We want to show that $g(C_2)=f(C_2)-2g(D_{10})-2g(D_6)-f(C_2^2)=f(C_2)-2f(D_{10})-2f(D_6)-f(C_2^2)+4$ is positive, where the second equality follows from the fact that $D_{10}$ and $D_6$ are maximal in $A_5$. 

For $R=C_2,D_{10},D_6$ and $C_2^2$ there is only one conjugacy class of subgroups isomorphic to $R$ in $K$, so by Corollary \ref{one conjugacy class}, we have
\[ f(R) = \frac{|N_T(R)|}{|N_K(R)|}. \]

It is easy to prove that $|N_K(C_2)|=4$, $|N_K(D_{10})|=10$, $|N_K(D_6)|=6$ and $|N_K(C_2^2)|=12$. 
Choose $\epsilon\in \{-1, 1\}$ such that $\frac{q+\epsilon}{2}$ is even. Then we have $N_T(C_2)=D_{q+\epsilon}$, so $|N_T(C_2)|\geqslant q-1$. By looking through the list of maximal subgroups of $T$, we see that $N_T(C_2^2)\leqslant C_2^2, A_4, S_4$ or $D_{q\pm 1}$. If $C_2^2\cong D_4 \leqslant N_T(D_4)\leqslant D_{q\pm 1}$, then $N_T(D_4)=N_{D_{q\pm 1}}(D_4)=D_4$ or $D_8$. So $|N_T(C_2^2)|\leqslant 24$. 
For $n=6$ or $10$, we use the list of maximal subgroups of $T$ given in \cite{PSLMoore} to see that $D_n\leqslant N_T(D_n)\leqslant D_{q\pm 1}, A_5$ or $S_4$. In the last two cases, if $D_n$ is a subgroup, it must be maximal, and hence $N_T(D_n)=D_n$. In the first case, we have $N_T(D_n)=N_{D_{q\pm 1}}(D_n)=D_n$ or $D_{2n}$. So $|N_T(D_n)|\leqslant 2n$ for $n=6$ and $10$. 

Putting this all together, we obtain $f(C_2)\geqslant \frac{q-1}{4}$, $f(C_2^2)\leqslant \frac{24}{12}=2$, $f(D_6)\leqslant \frac{12}{6}=2$ and $f(D_{10})\leqslant \frac{20}{10}=2$. It now follows that $g(C_2) \geqslant \frac{q-1}{4}-2\times 2-2\times 2-2+4>0$ for $q>25$. Taking into account the restrictions on $q$, it remains to consider the $q=19$ case in more detail. Here $f(C_2)=\frac{q+1}{4}=5$. Also, there is no $D_{12}$ in $T$, so $N_T(D_6)=D_6$ and $f(D_6)=1$. Then $g(C_2)\geqslant 5-2\times 2 - 2\times 1 - 2 + 4 > 0$, as desired. \end{proof}

\begin{lemma}\label{S_4} Suppose that $q\equiv \pm 1 \pmod{8}$ is a prime. Let $K\cong S_4$ be a subgroup of $T$. Then there exists $t\in T$ such that $K\cap K^t\cong C_2^2$. Moreover, if $q\geqslant 17$, then there exists $t\in T$ such that $K\cap K^t\cong C_2$. \end{lemma}

\begin{proof} Let $X$ be a subgroup of $K$ isomorphic to $C_2^2$ such that $X$ is not normal in $K$. Note that there is a unique $Y$ such that $X<Y<K$. Moreover, $Y\cong D_8$. Thus $g(X) = f(X) - f(D_8)$, which we want to show is positive.  

There is only one conjugacy class for $D_8$ in $K$, so by Corollary \ref{one conjugacy class} we have
\[ f(D_8) = \frac{|N_T(D_8)|}{|N_K(D_8)|}=\frac{|N_T(D_8)|}{8}. \]
From the list of maximal subgroups of $T$ we have $D_8\leqslant N_T(D_8)\leqslant D_{q\pm 1}$ or $S_4$. In the first case, we have $N_T(D_8)=N_{D_{q\pm 1}}(D_8)=D_8$ or $D_{16}$. In the second case, we have $N_T(D_8)=D_8$ as $D_8$ is maximal and not normal in $S_4$. So $|N_T(D_8)|\leqslant 16$. This implies that $f(D_8)\leqslant \frac{16}{8}=2$. 

We claim that $N_T(X)\cong S_4$. There are two conjugacy classes of $S_4$ in $T$, so let $J$ be a copy of $S_4$ not conjugate to $K$. Let $P$ be the $C_2^2$ in $J$ that is normal in $J$, and let $Q$ be the $C_2^2$ in $K$ that is normal in $K$. Now there are two conjugacy classes of $C_2^2$ in $T$. Since the normalisers of $P$ and $Q$ in $T$ ($J$ and $K$, respectively) are not conjugate, the subgroups $P$ and $Q$ cannot be conjugate. Hence $X$ is conjugate to either $P$ or $Q$. Thus the normaliser of $X$ is conjugate to the normaliser of either $P$ or $Q$, both of which are isomorphic to $S_4$. This proves that $N_T(X)\cong S_4$. 

The conjugacy class of $X$ in $K$ has size $3$, so the number of conjugates of $X$ in $T$ contained in $K$ is at least $3$. Then Lemma \ref{double counting} implies that 
\[ f(X)\geqslant \frac{3\cdot 24}{24}=3. \]

Putting this all together yields $g(X)\geqslant 3-2 > 0$. This proves the first part of the lemma.

For the second part of the lemma, we verify the $q=17$ case by a \textsc{Magma} ~\cite{Magma} calculation. We now show that if $q > 17$ then there exists $t\in T$ such that $K\cap K^t = Z$, where $Z$ is a subgroup of $X$ isomorphic to $C_2$. Since the only subgroups of $K$ that contain $Z$ as a maximal subgroup are $X$ and the copies of $S_3$, we need to show that $g(Z)=f(Z)-f(X)-2g(S_3)$ is positive. Since $S_3$ is maximal in $K$, we have $g(S_3)=f(S_3)-1$, and thus $g(Z)=f(Z)-f(X)-2f(S_3)+2$. 

There is only one conjugacy class for $Z$ in $T$, so the number of conjugates of $Z$ in $T$ that are contained in $K$ is equal to the number of copies of $C_2$ in $K$, which is equal to $9$. Now choose $\epsilon\in \{-1, 1\}$ such that $\frac{q+\epsilon}{2}$ is even. Then $N_T(Z)\cong D_{q+\epsilon}$. By Lemma \ref{double counting} we have 
\[ f(Z) = \frac{9(q+\epsilon)}{24}\geqslant \frac{3(q-1)}{8}. \]

There is only one conjugacy class for $S_3$ in $K$ so by Corollary \ref{one conjugacy class} we have
\[ f(S_3)=\frac{|N_T(S_3)|}{|N_K(S_3)|}=\frac{|N_T(S_3)|}{6}. \]
Now from the list of maximal subgroups of $T$ as given in \cite{PSLMoore} we have $S_3\leqslant N_T(S_3)\leqslant D_{q\pm 1}, S_4$ or $A_5$. In the first case, we have $N_T(S_3)=N_T(D_6)=N_{D_{q\pm 1}}(D_6)=D_6$ or $D_{12}$. In the second and third cases, we have $N_T(S_3)=S_3$ as $S_3$ is maximal and not normal in $S_4$ and $A_5$. So $|N_T(S_3)|\leqslant 12$. Thus $f(S_3)\leqslant \frac{12}{6}=2$.

There are four copies of $C_2^2$ in $K$ and earlier in this proof we calculated $|N_T(X)|=24$. So by Lemma \ref{double counting} we get
\[ f(X)\leqslant\frac{4\cdot 24}{24}=4. \]
So $g(Z)\geqslant \frac{3(q-1)}{8}-4-2\times 2 + 2 > 0$ for $q > 17$. Thus the second part of the lemma holds. \end{proof}

\begin{lemma}\label{dihedral intersect C_2} Let $q=2^f$ for some $f\geqslant 2$ and suppose that $K\cong D_{2(q+1)}$ is a subgroup of $T$. Then there exists $t\in T$ such that $K\cap K^t\cong C_2$. \end{lemma} 

\begin{proof} Let $y\in K$ be an involution. Since $q+1$ is odd all involutions of $K$ are conjugate and $\textbf{C}_K(y)=\langle y\rangle$. Moreover, all involutions in $T$ are conjugate and $T$ contains an elementary abelian subgroup of order $2^f\geqslant 4$. Hence there exists $t\in \textbf{C}_T(y)$ with $t\notin K$. Thus $y\in K\cap K^t\neq K$. We show that $K\cap K^t=\langle y\rangle$. Let $x\in K$ have odd order. Then by the maximality of $K$ in $T$ we have that $K=N_T(\langle x\rangle)$. If we also have $x\in K^t$ then by the same argument we have $K^t=N_T(\langle x\rangle)$, contradicting $K\neq K^t$. Thus $|K\cap K^t|$ is a power of $2$ and we are done. \end{proof}

\section{When $T=\text{PSL}(2,q)$}\label{T=PSL}

We will be working with the nonabelian simple group $T=\PSL(2,q)$, so for convenience let $G(m,q)=G(m,\text{PSL}(2,q))$ where $m\geqslant 2$. Recall $H=T\wr S_m$ and $L=\{(x,\ldots, x)\sigma\mid x\in T,\sigma\in S_m\}$.

The first lemma deals with the $m=2$ case. We shall write $S_2=\langle\iota\rangle$, so that $H=T\wr S_2=(T\times T)\rtimes\langle\iota\rangle$. 

\begin{lemma}\label{dihedral P_1} Suppose $m=2$. Let $D=P_1\wr S_2=(P_1\times P_1)\rtimes\langle\iota\rangle$ and suppose $t=(1,s)\in H$ such that $s\notin P_1$ is an involution. Then $D^t\cap L\cong D_{\frac{2(q-1)}{(2,q-1)}}$ and $\iota\notin D^t\cap L$. \end{lemma}

\begin{proof} We have $P_1\cap P_1^s\cong C_{\frac{q-1}{(2,q-1)}}$ and $\langle P_1\cap P_1^s,s\rangle\cong D_{\frac{2(q-1)}{(2,q-1)}}$. Let $M=\{ (x,x)\mid x\in T\}$. Then $D^t\cap M\cong P_1\cap P_1^s$. Moreover, $\iota^{(1,s)}=(s,s)\iota\in D^t\cap L$. Since $|L:M|=2$, it follows that $|D^t\cap L:D^t\cap M|=2$ and so $D^t\cap L=\langle \{(x,x)\mid x\in P_1\cap P_1^s\}, (s,s)\iota\rangle\cong \langle P_1\cap P_1^s,s\rangle\cong D_{\frac{2(q-1)}{(2,q-1)}}$. Moreover, note that $\iota\notin D^t\cap L$. \end{proof}

\begin{theorem}\label{subdegrees} Each row of Table \ref{list} gives a triple $(m,q,d)$ such that $d$ is a subdegree of $G(m,q)$. Unless stated otherwise, we assume $m\geqslant 2$ and $q\geqslant 4$. 

\begin{table}
\caption {Some subdegrees of $G(m,q)$} \label{list} 
\begin{center}
\begin{tabular}{ |c|c|c| } 
 \hline
 $m$ & $q$ & $d$ \\ 
 \hline\hline
$\geqslant 3$ & \text{all} & $(q+1)^m$\\
$2$ & $\equiv 1\pmod{4}$ & $(q+1)^2$\\    
 \hline                      
& $\equiv 1\pmod{4}$ & $(\frac{1}{2}q(q+1))^m$\\ 
& $\equiv 3\pmod{4}$ & $(\frac{1}{2}q(q-1))^m$ \\
 \hline
& even & $(\frac{1}{2}q(q-1))^m$ \\
 \hline
& odd & $(q(q-1))^m$ \\
& odd, $q\geqslant 7$ & $(q(q+1))^m$ \\
& even & $(q(q-1))^m$ and $(q(q+1))^m$\\
 \hline
& \text{all} & $\left (\frac{1}{(2,q-1)}(q^2-1)\right )^m$\\
 \hline
& $\equiv\pm 1\pmod{8}$ and prime & $\left (\frac{|T|}{24}\right )^m$ \\
 \hline
& $\equiv\pm 1\pmod{10}$, where $q$ prime & $\left (\frac{|T|}{60}\right )^m$\\
& or $q=p^2$ where $p\equiv\pm 3\pmod{10}$ prime & \\
 \hline
\end{tabular}
\end{center}
\end{table}
\end{theorem} 

\begin{proof} We begin by proving Row 1. When $m\geqslant 3$, this follows from letting $K=P_1$ in Corollary \ref{int with conj has centre} and noting that for any $s\in T\backslash P_1$, we have $P_1\cap P_1^s\cong C_{\frac{q-1}{(2,q-1)}}$, which has nontrivial centre. When $m=2$, set $D=(P_1\times P_1)\rtimes\langle\iota\rangle$ and note that $P_1$ is maximal in $T$. By Lemma \ref{dihedral P_1} there exists $t\in H$ such that $D^t\cap L\cong D_{q-1}$ and $\iota\notin D^t\cap L$. So $Z(D^t\cap L)\cong C_2$ contains an element $(\eta,\eta)\sigma$ with $\eta\neq 1$ and we can finish by Corollary \ref{max}.

Next, we prove Row 2. Let $\gamma\in T$ be an involution. By a \textsc{Magma} ~\cite{Magma} calculation and the fact that there is only one conjugacy class of involutions in $T$, if $q\equiv 1\pmod{4}$ then $\textbf{C}_T(\gamma)=D_{q-1}$ and if $q\equiv 3\pmod{4}$ then $\textbf{C}_T(\gamma)=D_{q+1}$. Hence Lemma \ref{centraliser} yields the subdegrees listed. 

Row 3 follows from letting $K=D_{2(q+1)}$ in Corollary \ref{C_2 int} and using Lemma \ref{dihedral intersect C_2} to guarantee the existence of $t\in T$ such that $K\cap K^t\cong C_2$.

Now consider Row 4. If $q$ is odd, then there exists $\gamma\in T$ of order $\frac{q+1}{2}$. If $q$ is odd and $q\geqslant 7$, then there exists $\gamma\in T$ of order $\frac{q-1}{2}$. Finally, if $q$ is even and $\epsilon\in\{-1,1\}$, then there exists $\gamma\in T$ of order $q+\epsilon$. In each of these cases it can be shown that $\textbf{C}_T(\gamma)=\langle\gamma\rangle$, so by Lemma \ref{centraliser} we are done.

To see why Row 5 holds, let $Z$ be the set of elements in $\text{SL}(2, q)$ that are in the centre of $\text{GL}(2, q)$ and set
\[ \gamma = \begin{bmatrix} 1&1\\ 0&1 \end{bmatrix} Z, \]
an element of $T$. Then $\textbf{C}_T(\gamma)$ has order $q$ and so Lemma \ref{centraliser} implies that $\left (\frac{|T|}{q}\right )^m=\left (\frac{q^2-1}{(2,q-1)}\right )^m$ is a subdegree of $G(m,q)$.

If $m\geqslant 3$, then Row 6 holds by letting $K=S_4$ in Corollary \ref{int with conj has centre}, and using Lemma \ref{S_4} which guarantees the existence of $t\in T$ such that $K\cap K^t\cong C_2^2$. It remains to consider the $m=2$ case. The $q=7$ case is done in \cite{CopSubPraeger}. If $q>7$, then the restrictions on $q$ imply that $q\geqslant 17$, and then by Lemma \ref{S_4} there exists $s\in T$ such that $K\cap K^s\cong C_2$. So by Corollary \ref{C_2 int} it follows that $|T:K|^2=\left (\frac{|T|}{24}\right )^2$ is a subdegree of $G(2,q)$.

If $q\geqslant 19$, then Row 7 follows from letting $K=A_5$ in Corollary \ref{C_2 int}, and using Lemma \ref{C_2} which guarantees the existence of $t\in T$ such that $K\cap K^t\cong C_2$. If $q=11$ and $m\geqslant 6$, then a \textsc{Magma} ~\cite{Magma} calculation finds $r,s\in T$ such that $K\cap K^r\cap K^s\cong C_2$. Then let $t=(1, \ldots, 1, r, r, s)=(a_1, \ldots, a_m)$. Note that neither of the elements $r, s$ or $sr^{-1}$ can be contained in $K$, as otherwise $K\cap K^r\cap K^s$ will be the intersection of two conjugates of $K$, which can never be $C_2$. Then
\begin{align*} D^t\cap L &= \{ (1, \ldots, 1, r^{-1}, r^{-1}, s^{-1})(k_1, \ldots, k_m)\sigma (1, \ldots, 1, r, r, s) \mid k_i\in K, \sigma\in S_m \} \cap L\\
&= \{ (1, \ldots, 1, r^{-1}, r^{-1}, s^{-1})(k_1, \ldots, k_m)(a_{1^{\sigma}}, \ldots, a_{m^{\sigma}})\sigma \mid k_i\in K, \sigma\in S_m \} \cap L. \end{align*}

Suppose $\sigma$ does not fix the set $\{1, \ldots, m-3\}$. Then there exist $i,j$ such that $1\leqslant i,j \leqslant m-3$ and $a_{i^{\sigma}} \neq a_{j^{\sigma}}$. For the product to be in $L$ we need $k_i a_{i^{\sigma}} = k_j a_{j^{\sigma}}$ and $a_{i^{\sigma}}a_{j^{\sigma}}^{-1} = k_i^{-1}k_j\in K$. But $a_{i^{\sigma}}a_{j^{\sigma}}^{-1}=r, s, sr^{-1}$ or the inverse of one of these, none of which lie in $K$, so the product cannot be contained in $L$.

Thus we need $\sigma$ to fix $\{1, \ldots, m-3\}$ setwise. Continuing this line of reasoning, we suppose $\sigma$ does not fix $\{m-2, m-1\}$ setwise. Then $a_{(m-2)^{\sigma}}\neq a_{(m-1)^{\sigma}}$. For the product to be in $L$ we need $r^{-1}k_{m-2}a_{(m-2)^{\sigma}}=r^{-1}k_{m-1}a_{(m-1)^{\sigma}}$ and $a_{(m-2)^{\sigma}}a_{(m-1)^{\sigma}}^{-1} = k_{m-2}^{-1}k_{m-1}\in K$. But this means that $sr^{-1}$ or $rs^{-1}=(sr^{-1})^{-1}$ lies in $K$, which is not true, so $\sigma$ must fix $\{m-2, m-1\}$ setwise. Let $P\cong C_2\times S_{m-3}$ be the subgroup of $S_m$ consisting of all permutations fixing $\{1, \ldots, m-3\}$ and $\{m-2, m-1\}$ setwise. Then 
\begin{align*} D^t \cap L &= \{ (k_1, \ldots, k_{m-3}, k_{m-2}^r, k_{m-1}^r, k_m^s)\sigma \mid k_i\in K, \sigma\in P \}\cap L\\
&\cong (K\cap K^r\cap K^s) \times P.\end{align*} 
Since $K\cap K^r\cap K^s=C_2$, it follows that $Z(D^t\cap L)$ contains an element $(\eta,\ldots,\eta)\sigma$ with $\eta\neq 1$. Thus Corollary \ref{max} implies $|T:K|^m = \left (\frac{|T|}{60}\right )^m$ is a subdegree of $G(m,q)$.

The $q=11$ and $m=2,3,4$ or $5$ cases can be done via computations in \textsc{Magma} ~\cite{Magma}, by finding some $t\in H$ such that $Z(D^t\cap L)$ contains a suitable element to apply Corollary \ref{max}. \end{proof} 

We can also make use of Lemma \ref{fn} by allowing $D$ not to be maximal in $H$. In this case we do not get an exact subdegree, but we do prove that there exists a subdegree dividing some number. This will prove to be useful in constructing an infinite family with nontrivial coprime subdegrees.

\begin{lemma}\label{D not maximal} Let $D$ be a subgroup of $H$. Suppose there exists $t\in H$ such that $(\eta, \ldots, \eta)\sigma\in Z(D^t \cap L)$ with $\eta\neq 1$. Then $G(m,T)$ has a nontrivial subdegree dividing $|H:D|$. \end{lemma}

\begin{proof} Define $g$ as in Equation \eqref{weird fn}. Then by Lemma \ref{fn}, $D\leqslant \textbf{C}_H(g)$, so the subdegree $|H:\textbf{C}_H(g)|$ divides $|H:D|$. To show that this subdegree is nontrivial, it suffices to show that $g\neq\text{id}$, which is true as by Lemma \ref{fn}, we have that $g$ is nonconstant. \end{proof} 

\begin{corollary}\label{P_1P_1 not maximal} The group $G(2,q)$ has a nontrivial subdegree dividing $2(q+1)^2$. \end{corollary}

\begin{proof} Let $D=P_1\times P_1$ and let $t=(1,s)$ for some $s\notin P_1$. Then $D^t\cap L\cong P_1\cap P_1^s\cong C_{\frac{q-1}{(2,q-1)}}$ has nontrivial centre. Since $D^t\cap L\leqslant T^2$, it follows that the conditions in Lemma \ref{D not maximal} have been satisfied, so $G(2,q)$ has a subdegree dividing $|H:D|=2(q+1)^2$. \end{proof} 

We now present some infinite families of primitive TW groups with nontrivial coprime subdegrees.

\begin{theorem}\label{infinite families} Each row of Table \ref{coprime subdegrees} gives a quadruple $(m,q,r,d)$ such that $r$ and $d$ are nontrivial coprime subdegrees of $G(m,q)$. 

\begin{table}
\caption {Nontrivial coprime subdegrees of $G(m,q)$} \label{coprime subdegrees} 
\begin{center}
\begin{tabular}{ |c|c|c|c| } 
 \hline
 $m$ & $q$ & $r$ & $d$ \\ 
 \hline\hline
$2$ & $\equiv 3\pmod{4}$ & $(\frac{1}{2}q(q-1))^2$ & divides $2(q+1)^2$\\    
 \hline                      
$\geqslant 3$ & $\equiv 3\pmod{4}$ & $(\frac{1}{2}q(q-1))^m$ & $(q+1)^m$\\ 
 \hline
$\geqslant 3$ & even & $(\frac{1}{2}q(q-1))^m$ or $(q(q-1))^m$ & $(q+1)^m$ \\
 \hline
& $29$ & $30^m$ & $203^m$ \\
 \hline
& $7$ & $7^m$ & $24^m$ \\
 \hline
& $11$ & $11^m$ & $60^m$\\
 \hline
\end{tabular}
\end{center}
\end{table}
\end{theorem}

\begin{proof} This follows from using the data in Lemma \ref{subdegrees}, apart from using Corollary \ref{P_1P_1 not maximal} to construct the subdegree dividing $2(q+1)^2$ in Row 1. \end{proof}

\section{Characterisation results}\label{any m at least 2}

Recall the action of $H$ on $N=\{f\in T^H\mid f(z\ell)=f(z)^{\phi(\ell)}\quad\forall z\in H, \ell\in L\}$, where $H=T\wr S_m$, $L=\{(x,\ldots, x)\sigma\mid x\in T, \sigma\in S_m\}$, and $\phi((x,\ldots,x)\sigma)=i_x$ for all $(x,\ldots,x)\sigma\in L$ ($i_x$ denotes the automorphism induced on $T$ by conjugation by $x$). We will investigate this action in more detail. Also, define the projection maps $\pi_i:T^m\rightarrow T$ for all $1\leqslant i\leqslant m$ by $\pi_i((t_1,\ldots,t_m))=t_i$ for all $t_1,\ldots,t_m\in T$. 

\begin{lemma}\label{equivalent N} Let $f\in T^H$. Then $f\in N$ if and only if
\[ f((t_1,\ldots,t_m)\sigma)=[f((t_1t_m^{-1},\ldots,t_{m-1}t_m^{-1},1))]^{t_m} \]
for all $t_1,\ldots,t_m\in T$ and $\sigma\in S_m$. \end{lemma}

\begin{proof} Let $f\in N$. For all $t_1,\ldots,t_m\in T$ and $\sigma\in S_m$ with $z=(t_1t_m^{-1},\ldots,t_{m-1}t_m^{-1},1)$ and $\ell=(t_m,\ldots,t_m)\sigma$ we have
\begin{align*} f((t_1,\ldots,t_m)\sigma)&=f(z\ell)\\&=f(z)^{\phi(\ell)}\\&=[f((t_1t_m^{-1},\ldots,t_{m-1}t_m^{-1},1))]^{t_m}. \end{align*}

Conversely, suppose that $f((t_1,\ldots,t_m)\sigma)=[f((t_1t_m^{-1},\ldots,t_{m-1}t_m^{-1},1))]^{t_m}$ holds for all $t_1,\ldots,t_m\in T$ and $\sigma\in S_m$. Then for any $z=(a_1,\ldots,a_m)\sigma_1\in H$ and $\ell=(t,\ldots,t)\sigma_2$ we have
\begin{align*} f(z\ell)=f((a_1t,\ldots,a_mt)\sigma_1\sigma_2)&=[f((a_1a_m^{-1},\ldots,a_{m-1}a_m^{-1},1))]^{a_mt}\\&=[f((a_1,\ldots,a_m)\sigma_1)]^t\\&=f(z)^{\phi(\ell)}, \end{align*}
so $f\in N$. \end{proof}

\begin{lemma}\label{no T^m} If for some $f\in N$ we have $T^m\leqslant H_f$, then $f=\text{id}$. \end{lemma}

\begin{proof} Define a function $\alpha:T^{m-1}\rightarrow T$ by $\alpha(x_1,\ldots,x_{m-1})=f((x_1,\ldots,x_{m-1},1))$ for all $x_i\in T$. We will use Lemma \ref{equivalent N} in the following series of manipulations. We have $(t_1,\ldots,t_m)\in H_f$ for all $t_1,\ldots,t_m\in T$ if and only if for all $t_i,a_i\in T$ and $\sigma\in S_m$ we have
\begin{align*} &f^{(t_1,\ldots,t_m)}((a_1,\ldots,a_m)\sigma)=f((a_1,\ldots,a_m)\sigma)\\
&\iff f((t_1a_1,\ldots,t_ma_m)\sigma)=f((a_1,\ldots,a_m)\sigma)\\
&\iff [f((t_1a_1a_m^{-1}t_m^{-1},\ldots,t_{m-1}a_{m-1}a_m^{-1}t_m^{-1},1))]^{t_m}=f((a_1a_m^{-1},\ldots,a_{m-1}a_m^{-1},1)).\end{align*}
This is equivalent to
\[ [\alpha(t_1s_1t_m^{-1},\ldots,t_{m-1}s_{m-1}t_m^{-1})]^{t_m}=\alpha((s_1,\ldots,s_{m-1}))\quad\forall t_i,s_i\in T.\] 
By setting $t_m$ and the $s_i$ to be $1$, and varying $t_1,\ldots,t_{m-1}$, we see that $\alpha$ is a constant function. Suppose the value of $\alpha$ is always equal to $y\in T$. Then $y=y^{t_m}$ for all $t_m\in T$, so $y\in Z(T)=1$. Thus $\alpha(x_1,\ldots,x_{m-1})=1$ for all $x_i\in T$. By Lemma \ref{equivalent N}, we have $f=\text{id}$. \end{proof}

\begin{lemma}\label{T^m not in X} Let $X$ be a subgroup of $H=T\wr S_m$ that does not contain $T^m$. Then $|H:X|$ is divisible by $|T:K|$ for some maximal subgroup $K$ of $T$. \end{lemma}

\begin{proof} Since $T^m$ is normal in $H$, it follows that $|X:X\cap T^m|$ divides $|H:T^m|$. Now
\[ |H:X|=\frac{|H:X\cap T^m|}{|X:X\cap T^m|}=\frac{|H:T^m||T^m:X\cap T^m|}{|X:X\cap T^m|}, \]
so $|H:X|$ is divisible by $|T^m:X\cap T^m|$. It suffices to show that for any proper subgroup $Y$ of $T^m$, we have that $|T^m:Y|$ is divisible by $|T:K|$ for some maximal subgroup $K$ of $T$. 

Suppose there exists $i$ such that $\pi_i(Y)<T$, and assume without loss of generality that $\pi_1(Y)<T$. Then there exists a maximal subgroup $K$ of $T$ such that $\pi_1(Y)\leqslant K$. So $Y\leqslant K\times T\times\cdots\times T$, and $|T^m:Y|$ is divisible by $|T^m:K\times T\times\cdots\times T|=|T:K|$. Thus we are done in this case. If $\pi_i(Y)=T$ for all $i$, then $Y\cong T^k$ for some $k<m$. Then $|T^m:Y|=|T|^{m-k}$ is divisible by $|T|$ and thus $|T:K|$ for any maximal subgroup $K$ of $T$. \end{proof} 

\begin{corollary}\label{divisible by |T:K|} Let $f\in N\backslash\{\text{id}\}$. Then $|H:H_f|$ is divisible by $|T:K|$ for some maximal subgroup $K$ of $T$. \end{corollary}

\begin{proof} This follows immediately from Lemma \ref{no T^m} and Lemma \ref{T^m not in X}. \end{proof} 

In the $m\geqslant 3$ case, the results we have from the previous sections are enough to determine all $q$ such that $G(m,q)$ has nontrivial coprime subdegrees. 

\begin{theorem} Suppose $m\geqslant 3$. Then $G(m,q)$ has nontrivial coprime subdegrees if and only if either $q$ is even or $q\equiv 3\pmod{4}$ or $q=29$. \end{theorem}

\begin{proof} By the results in Section \ref{T=PSL} we see that for all these values of $q$ the group $G(m,q)$ has nontrivial coprime subdegrees. Now suppose that the group $G(m,q)$ has nontrivial coprime subdegrees. Then there exist $f,g\in N\backslash\{\text{id}\}$ such that $|H:H_f|$ and $|H:H_g|$ are coprime. By Corollary \ref{divisible by |T:K|} there exist $K_1$ and $K_2$ maximal in $T$ such that $|H:H_f|$ is divisible by $|T:K_1|$ and $|H:H_g|$ is divisible by $|T:K_2|$. Since $|T:K_1|$ and $|T:K_2|$ are coprime, it follows from Lemma 3.16 in \cite{Isaacs} that $T=K_1K_2$ is a maximal coprime factorisation. By the list in \cite{CopSubPraeger} it follows that either $q$ is even or $q\equiv 3\pmod{4}$ or $q=29$. \end{proof} 

\section{Maximal subgroups of $H=T\wr S_2$ and $T\times T$}

To address the $m=2$ case, we will need information about the maximal subgroups of $H=T\wr S_2$ and $T\times T$, where $T$ is a nonabelian simple group. 

\begin{lemma}\label{maximal H} Let $X$ be a maximal subgroup of $H=T\wr S_2=(T\times T)\rtimes\langle\iota\rangle$. Then up to conjugacy $X$ has one of the following three types:
\begin{enumerate}
\item $X=T^2=T\times T$ and $|H:X|=2$; 
\item $X=\langle S, (a,b)\iota\rangle$ where $S=\{(t,t^{\sigma})\mid t\in T\}$ for some $\sigma\in \text{Aut}(T)$ and $a,b\in T$ such that $(ab)^{\sigma}=ba$ and $\sigma i_b \sigma = i_a$, and $|H:X|=|T|$; 
\item $X=K\wr S_2$ for some $K$ maximal in $T$, and $|H:X|=|T:K|^2$. 
\end{enumerate}
\end{lemma} 

\begin{proof} Let $M$ be a proper subgroup of $H$. If $M\leqslant T^2$, then $M$ is contained in the proper subgroup $T^2$, which is of type (1).

If $M\nleqslant T^2$, there exists $(a,b)\iota\in M$ for some $a,b\in T$. Thus $|M:M\cap T^2|=2$ and $\langle M\cap T^2, (a,b)\iota\rangle=M$. Also note that $((a,b)\iota)^2=(ab,ba)\in M\cap T^2$. Moreover, for any $t_1,t_2\in T$ we have $(t_1,t_2)^{(a,b)\iota}=(t_2^b,t_1^a)$, so $\pi_2(M\cap T^2)=(\pi_1(M\cap T^2))^a$. 

If $\pi_1(M\cap T^2)=\pi_2(M\cap T^2)=T$, then $M\cap T^2=T^2$ or $M\cap T^2\cong T$. If $M\cap T^2=T^2$, then since $M$ contains an element outside of $T^2$ we must have $M=H$, a contradiction. So $M\cap T^2\cong T$. Thus $M\cap T^2=\{(t,t^{\sigma})\mid t\in T\}=S$ where $\alpha\in\text{Aut}(T)$. Since $(ab,ba)\in M\cap T^2$, it follows that $(ab)^{\sigma}=ba$. Since $M\cap T^2\vartriangleleft M$, we have that $(a,b)\iota$ normalises $M\cap T^2$. If we let $i_x$ denote the automorphism of $T$ induced by conjugation by $x$, we have $(t,t^{\sigma})^{(a,b)\iota}=(t^{\sigma i_b},t^{i_a}) \in M\cap T^2$ for all $t\in T$, so $\sigma i_b \sigma = i_a$. Thus $M$ is a subgroup of type (2) and since $|M:M\cap T^2|=2$ we have that $|H:M|=|T|$.  

Finally, if $\pi_1(M\cap T^2)<T$, let $K$ be maximal in $T$ such that $\pi_1(M\cap T^2)\leqslant K$. Then $\pi_2(M\cap T^2)\leqslant K^a$. So $M=\langle M\cap T^2, (a,b)\iota\rangle\leqslant\langle K\times K^a, (a,b)\iota\rangle$. Moreover, since $(ab,ba)\in M\cap T^2$ we have that $ab\in K$. Then $(1, ((ab)^{-1})^a)(a,b)\iota=(a,a^{-1})\iota$, so $\langle K\times K^a, (a,b)\iota\rangle=\langle K\times K^a, (a,a^{-1})\iota\rangle$. Now $((K\times K)\rtimes\langle\iota\rangle)^{(1,a)}=\langle K\times K^a, (a,a^{-1})\iota\rangle$, so up to conjugacy $K$ is contained in a subgroup of type (3). 

Hence each subgroup $M$ of $H$ is contained in a type (1), (2) or (3) subgroup. Observe that a type $(i)$ subgroup cannot be contained in a type $(j)$ subgroup if $i\neq j$, and that a type $(i)$ subgroup cannot be properly contained in another type $(i)$ subgroup. Hence the subgroups stated in the lemma are precisely the maximal subgroups of $H$. \end{proof}


\begin{lemma}\label{maximal T^2} The maximal subgroups of $T^2$ are of one of the following two types:
\begin{enumerate}
\item $K\times T$ or $T\times K$ for some $K$ maximal in $T$;
\item $\{(t,t^{\sigma})\mid t\in T\}$ where $\sigma\in\text{Aut}(T)$. 
\end{enumerate}
\end{lemma}

\begin{proof} Let $M$ be a proper subgroup of $T^2$. If $\pi_1(M)<T$, then let $K$ be maximal in $T$ such that $\pi_1(M)\leqslant K$. Then $M\leqslant K\times T$. Similarly, if $\pi_2(M)<T$ then there exists $K$ maximal in $T$ such that $M\leqslant T\times K$. So if there exists $i$ such that $\pi_i(M)$ is a proper subgroup of $T$, it follows that $M$ is contained in a type (1) subgroup. If $\pi_1(M)=\pi_2(M)=T$, then $M\cong T$, so $M$ is equal to (and thus contained in) a type (2) subgroup. So in all cases, $M$ is contained in a type (1) or (2) subgroup.

It is easy to show that a type $(i)$ subgroup cannot be contained in a type $(j)$ subgroup if $i\neq j$, and that a type $(i)$ subgroup cannot be properly contained in another type $(i)$ subgroup. So type (1) and type (2) subgroups are maximal in $T^2$. \end{proof}


\section{The $m=2$ case}\label{m=2}

We now address the $m=2$ case. Since $\text{PSL}(2,4)\cong\text{PSL}(2,5)$, it follows that $G(2,4)\cong G(2,5)$. So from now on, suppose $q\neq 5$. Recall that $H=T\wr S_2$ and $L=\{(x,x)\sigma\mid x\in T, \sigma\in S_2\}$. 

\begin{lemma}\label{two cases} Suppose that $|H:H_f|$ and $|H:H_g|$ are nontrivial coprime subdegrees of $G(2,q)$. Let $R_1$ and $R_2$ be maximal subgroups of $H$ such that $H_f\leqslant R_1$ and $H_g\leqslant R_2$. Then one of the following holds:
\begin{enumerate}
\item $R_1$ and $R_2$ are of type (3) in Lemma \ref{maximal H};
\item after reordering, $R_1=T^2$ and $H_f$ is contained in a subgroup of type (1) in Lemma \ref{maximal T^2}, $R_2$ is of type (3) in Lemma \ref{maximal H} and $|H:R_2|$ is odd. 
\end{enumerate}
\end{lemma}

\begin{proof} By Lemma \ref{divisible by |T:K|}, neither $R_1$ nor $R_2$ can be of type (2) as $|T|$ and $|T:K|$ are never coprime for any $K$ maximal in $T$. If both $R_1$ and $R_2$ are of type (3) in Lemma \ref{maximal H}, we obtain the first case. Now assume without loss of generality that $R_1$ is of type (1) and so $R_1=T^2$. Then by Lemma \ref{no T^m} we have $H_f\neq T^2$, so $H_f$ is strictly contained in $T^2$. Thus $H_f$ is contained in a maximal subgroup of $T^2$, say $M$. Then $|H:H_f|$ is divisible by $|H:M|$. If $M$ is of type (2) in Lemma \ref{maximal T^2}, then $|H:M|=2|T|$, so by Lemma \ref{divisible by |T:K|} it is not possible for $|H:H_f|$ and $|H:H_g|$ to be coprime. Thus $M$ is of type (1) in Lemma \ref{maximal T^2}. It is not possible for $R_2$ to be of type (1) or (2) as then both $|H:R_1|$ and $|H:R_2|$ will be even. Thus $R_2$ is of type (3) in Lemma \ref{maximal H}. We note that $|H:R_2|$ is odd since $|H:M|$ is even. \end{proof} 

\begin{lemma}\label{further details} Suppose that Case (2) of Lemma \ref{two cases} holds. Let $H_f$ be contained in a subgroup of type (1) in Lemma \ref{maximal T^2} constructed from $K_1$. Let $R_2$ be constructed from $K_2$. Then the possibilities for $K_1$ and $K_2$ are as follows:
\begin{enumerate} 
\item $q$ is even, $K_1=D_{2(q+1)}$ and $K_2=P_1$;
\item $q\equiv 3\pmod{4}$ with $q>7$, $K_1=P_1$ and $K_2=D_{q+1}$;
\item $q\in\{7,23\}$, $K_1=P_1$ and $K_2=S_4$;
\item $q\in\{11,19,29,59\}$, $K_1=P_1$ and $K_2=A_5$. 
\end{enumerate}
\end{lemma}

\begin{proof} We need $|H:M|=2|T:K_1|$ to be coprime to $|H:R_2|=|T:K_2|^2$, where $H_f\leqslant M=T\times K_1$ or $K_1\times T$. In particular $|T:K_1|$ must be coprime to $|T:K_2|$ so we can systematically work through the list of maximal coprime factorisations in \cite{CopSubPraeger}, to obtain the possibilities mentioned above. \end{proof}

We will do some detailed analysis about the possibilities for $H_f$ and $H_g$ when $|H:H_f|$ and $|H:H_g|$ are nontrivial coprime subdegrees of $G(2,q)$. Since $H_f^h=H_{f^h}$ for all $h\in H$, it suffices to consider $H_f$ up to conjugacy in $H$. We begin with a lemma about subgroups of $T\times P_1$ that will help simplify our casework later.

\begin{lemma}\label{T P_1} Let $q\equiv 3\pmod{4}$. Let $X$ be a subgroup of $M=T\times P_1$ such that $|M:X|$ is coprime to $\frac{1}{2}q(q-1)$. Then $X=M$ or $X=P_1^t\times P_1$ for some $t\in T$. \end{lemma}

\begin{proof} If $\pi_1(X)=T$, then $T=\pi_1(X)\cong X/\text{ker}\pi_1$, so $|X|=|T||\text{ker}\pi_1|$. Now $|M:X|=\frac{|T||P_1|}{|T||\text{ker}\pi_1|}=\frac{|P_1|}{|\text{ker}\pi_1|}$ divides $|P_1|=\frac{1}{2}q(q-1)$, so we must have $X=M$. 

Now suppose $\pi_1(X)<T$. Then $|X|\leqslant |\pi_1(X)||P_1|$. Note that $|M|=|T||P_1|=|P_1|^2(q+1)$, and every prime factor of $|P_1|^2$ is clearly a prime factor of $\frac{1}{2}q(q-1)=|P_1|$, so for $|M:X|$ to be coprime to $\frac{1}{2}q(q-1)$ we require $|X|\geqslant |P_1|^2$. Combining this with $|X|\leqslant |\pi_1(X)||P_1|$ we obtain $|\pi_1(X)|\geqslant |P_1|$. We cannot have $q=7$ and $\pi_1(X)=S_4$ as then $|M:X|$ is divisible by $|M:S_4\times P_1|=7$, which is not coprime to $\frac{1}{2}q(q-1)=21$. Similarly, we cannot have $q=11$ and $\pi_1(X)=A_5$ as then $|M:X|$ is divisible by $|M:A_5\times P_1|=11$ which is not coprime to $\frac{1}{2}q(q-1)=55$. So it follows from \cite{PSLMoore} that $\pi_1(X)\cong P_1$. Since $|X|\geqslant |P_1|^2$ it follows that $X=P_1^t\times P_1$. \end{proof} 

\begin{lemma}\label{nine cases} If $|H:H_f|$ and $|H:H_g|$ are nontrivial coprime subdegrees of $G(2,q)$, then up to reordering one of the possibilities in Table \ref{big} holds. 

\begin{table}
\caption {Various possibilities for $H_f$ and $H_g$} \label{big} 
\begin{center}
\begin{tabular}{ |c|c|c| } 
 \hline
 $q$ & $H_f$ & $H_g$ \\ 
 \hline\hline
 even & $P_1\wr S_2$ & $C_{q+1}\times C_{q+1}\leqslant H_g$ \\                        
& $P_1\wr S_2$ & $|T\times D_{2(q+1)}:H_g|$ is coprime to $(q+1)^2$\\ 
 \hline
 $\equiv 3\pmod{4}$ & $P_1\wr S_2,T\times P_1$ or $P_1\times P_1$ & $D_{q+1}\wr S_2$ \\
 \hline
$7$ & $C_7\times C_7\leqslant H_f$ & $S_4\wr S_2$\\
& $P_1\wr S_2,T\times P_1$ or $P_1\times P_1$ & $D_8\wr S_2$ \\
 \hline
$11$ & $C_{11}\times C_{11}\leqslant H_f$ & $A_5\wr S_2$\\
& $P_1\wr S_2,T\times P_1$ or $P_1\times P_1$ & $A_4\wr S_2$ \\
 \hline
 $19$ & $P_1\wr S_2,T\times P_1$ or $P_1\times P_1$ & $A_5\wr S_2$ \\
 \hline
 $23$ & $P_1\wr S_2,T\times P_1$ or $P_1\times P_1$ & $S_4\wr S_2$ \\
 \hline
$29$ & $X\times X\leqslant H_f$ where $|P_1:X|=2$ & $A_5\wr S_2$\\ 
 \hline 
$59$ & $P_1\wr S_2,T\times P_1$ or $P_1\times P_1$ & $A_5\wr S_2$ \\
 \hline
\end{tabular}
\end{center}
\end{table}
\end{lemma} 

\begin{proof} Suppose that $|H:H_f|$ and $|H:H_g|$ are nontrivial coprime subdegrees of $G(2,q)$. Let $R_1$ and $R_2$ be maximal subgroups of $H$ such that $H_f\leqslant R_1$ and $H_g\leqslant R_2$. According to Lemma \ref{two cases}, we can split our analysis into two cases. 

\textbf{Case 1}: $R_1$ and $R_2$ are of type (3) in Lemma \ref{maximal H}, and are constructed from $K_1$ and $K_2$, respectively. 

Then $|T:K_1|^2$ and $|T:K_2|^2$ are the indices of $R_1$ and $R_2$ so they must be coprime. Thus $|T:K_1|$ and $|T:K_2|$ are coprime, so $T=K_1K_2$ is a maximal coprime factorisation by Lemma 3.16 in \cite{Isaacs}. We will work through the list in \cite{CopSubPraeger}, and assume without loss of generality that $K_1=P_1$. 

\textbf{Subcase 1a}: $q$ is even and $K_2=D_{2(q+1)}$. 

Since $|R_1:H_f|$ is coprime to $|H:R_2|=(\frac{1}{2}q(q-1))^2$ and divides $|R_1|=2(q(q-1))^2$, it follows that $H_f=R_1$. Since $|H:R_1|=(q+1)^2$, we have that $|H:H_g|=|H:R_2||R_2:H_g|=(\frac{1}{2}q(q-1))^2|R_2:H_g|$ is coprime to $(q+1)^2$. Thus $|R_2:H_g|$ is coprime to $(q+1)^2$ so $C_{q+1}\times C_{q+1}\leqslant H_g$ and we are in Row 1 of Table \ref{big}.

\textbf{Subcase 1b}: $q\equiv 3\pmod{4}$ with $q>7$ and $K_2=D_{q+1}$. 

We have $|H|=|R_1||R_2|$, so $H_f=R_1$ and $H_g=R_2$. Thus we are in Row 2. 

\textbf{Subcase 1c}: $q\in\{7,23\}$ and $K_2=S_4$. 

Suppose $q=7$. Since $|R_2:H_g|$ is coprime to $|H:R_1|=8^2$ and divides $|R_2|=2(24^2)$, it follows that $|R_2:H_g|$ is a power of $3$. Thus $H_g=R_2$ or $H_g\cong D_8\wr S_2$. In either case, $|R_1:H_f|$ is coprime to $|H:R_2|=7^2$, so $H_f$ contains the subgroup of $R_1$ isomorphic to $C_7\times C_7$. If $H_g=R_2$ we are in Row 3. Furthermore, if $H_g=D_8\wr S_2$, then since $|R_1:H_f|$ is coprime to $|H:H_g|=21^2$ and divides $|R_1|=2(21^2)$, we must have $|R_1:H_f|=1$ or $2$. Thus $H_f=R_1$ or $H_f=P_1\times P_1$ and we are in Row 3. 

If $q=23$, we have $|H|=|R_1||R_2|$, so $H_f=R_1$ and $H_g=R_2$. Thus we are in Row 6. 

\textbf{Subcase 1d}: $q\in\{11,19,29,59\}$ and $K_2=A_5$.  

If $q=11$, then since $|R_2:H_g|$ is coprime to $|H:R_1|=12^2$ and divides $|R_2|=2(60^2)$, it follows that $|R_2:H_g|$ must be a power of $5$. Thus $H_g=R_2$ or $H_g=A_4\wr S_2$. In either case, $|R_1:H_f|$ is coprime to $|H:R_2|=11^2$, so $H_f$ contains the subgroup of $R_1$ isomorphic to $C_{11}\times C_{11}$. If $H_g=R_2$ we are in Row 4. Furthermore, if $H_g=A_4\wr S_2$, then $|R_1:H_f|$ is coprime to $|H:H_g|=55^2$ and divides $|R_1|=2(55^2)$, so $|R_1:H_f|=1$ or $2$. Thus $H_f=R_1$ or $H_f=P_1\times P_1$ and we are in Row 4.  

Suppose $q=19$. Then $|R_1:H_f|$ divides $|R_1|=2(171^2)$ and is coprime to $|H:R_2|=57^2$, so $|R_1:H_f|=1$ or $2$. This means that $H_f=R_1$ or $H_f=P_1\times P_1$. Also, $|R_2:H_g|$ divides $|R_2|=2(60^2)$ and is coprime to $|H:R_1|=20^2$, so $|R_2:H_g|$ is a power of $3$. This can only happen if $H_g=R_2$, so we are in Row 5. 

If $q=29$, then since $|R_2:H_g|$ is coprime to $|H:R_1|=30^2$ and divides $|R_2|=2(60^2)$, we must have $H_g=R_2$. Then since $|R_1:H_f|$ is coprime to $|H:R_2|=203^2$ and divides $|R_1|=2(406^2)$, it follows that $|R_1:H_f|$ is a power of $2$. This can only happen if $H_f$ contains the subgroup $X\times X$ of $R_1$, where $X$ is the index $2$ subgroup of $P_1$, and we get Row 7. 

If $q=59$, we are again in the situation when $|H|=|R_1||R_2|$, so $H_f=R_1$ and $H_g=R_2$. Thus we are in Row 8. 

\textbf{Case 2}: $R_1=T^2$ and $H_f$ is contained in a subgroup $M$ of type (1) in Lemma \ref{maximal T^2}, and $R_2$ is of type (3) in Lemma \ref{maximal H}.

Suppose that $M$ is constructed from $K_1$ and $R_2$ is constructed from $K_2$. We work through the possibilities for $K_1$ and $K_2$ outlined in Lemma \ref{further details}.

\textbf{Subcase 2a}: $q$ is even, $K_1=D_{2(q+1)}$ and $K_2=P_1$. 

Since $|R_2:H_g|$ divides $|R_2|=2(q(q-1))^2$ and is coprime to $|H:M|=q(q-1)$, it follows that $H_g=R_2$. Since $|H:R_2|=(q+1)^2$, we have that $|H:H_f|=|H:M||M:H_f|$ is coprime to $(q+1)^2$. Thus $|M:H_f|$ is coprime to $(q+1)^2$. Thus we are in Row 1 by interchanging $f$ and $g$. 

\textbf{Subcase 2b}: Either one of Cases 2, 3 or 4 in Lemma \ref{further details} holds. 

Then $K_1=P_1$ and $M=T\times P_1$. Hence $|H:H_f|=|H:M||M:H_f|=2(q+1)|M:H_f|$ is coprime to $|H:H_g|=|H:R_2||R_2:H_g|$, and so $2(q+1)$ is coprime to $|R_2:H_g|$. Now $|R_2:H_g|$ divides $|R_2|=2|K_2|^2$, and we can check that in all cases apart from Case 3 with $q=7$ and Case 4 with $q\in\{11,19\}$, the prime factors of $2|K_2|^2$ are the same as the prime factors of $2(q+1)$. So if we are in Case 2, Case 3 with $q=23$, or Case 4 with $q\in\{29,59\}$, we must have $H_g=R_2$. We claim this is also true in Case 4 with $q=19$. Indeed, $2(q+1)=20$ is coprime to $|R_2:H_g|$ which divides $|R_2|=2(60^2)$, so $|R_2:H_g|$ is a power of $3$, which can only happen if $H_g=R_2$.

Suppose we are in either Case 2, Case 3 with $q=23$, or Case 4 with $q\in\{19,59\}$. Then from the above argument we have $H_g=R_2=X\wr S_2$ and $H_f\leqslant M=T\times P_1$. So $|H:H_g|=|T:X|^2$ is coprime to $|M:H_f|$. We can check that in all these cases the prime factors of $|T:X|$ are the same as the prime factors of $\frac{1}{2}q(q-1)$, so by Lemma \ref{T P_1} we have $H_f=P_1\times P_1$ or $H_f=M$. Thus we are in Rows 2,5,6 and 8. 

Consider Case $4$ with $q=29$. We have that $|M:H_f|$ is coprime to $|H:H_g|=203^2$. Moreover, $|M|=203^2\cdot 120$ so $|M:H_f|$ divides $120$. This can only happen if $H_f$ contains a subgroup of $M$ isomorphic to $X\times X$ where $X$ has index $2$ in $P_1$, yielding Row 7.

Next, consider Case 3 with $q=7$. Then $|R_2:H_g|$ is coprime to $2(q+1)=16$ and divides $|R_2|=2(24^2)$, so $|R_2:H_g|$ is a power of $3$. This means that $H_g=R_2$ or $H_g\cong D_8\wr S_2$. In either case, $|M:H_f|$ is coprime to $|H:R_2|=7^2$, so $H_f$ contains a subgroup of $M$ isomorphic to $C_7\times C_7$. If $H_g=D_8\wr S_2$, then $|M:H_f|$ is coprime to $|H:H_g|=21^2$ and divides $|M|=21^2\cdot 7$, so $|M:H_f|$ is a power of $2$ and $H_f=M$ or $H_f=P_1\times P_1$. Thus we are in Row 3. 

Finally, consider Case 4 with $q=11$. Then $|R_2:H_g|$ is coprime to $2(q+1)=24$ and divides $|R_2|=2(60^2)$, so $|R_2:H_g|$ is a power of $5$. This means that $H_g=R_2$ or $H_g=A_4\wr S_2$. In either case, $|M:H_f|$ is coprime to $|H:R_2|=11^2$, so $H_f$ contains a subgroup of $M$ isomorphic to $C_{11}\times C_{11}$. If $H_g=A_4\wr S_2$, then $|M:H_f|$ is coprime to $|H:H_g|=55^2$ and divides $|M|=55^2\cdot 12$, it follows that $|M:H_f|$ divides $12$. This can only happen if $H_f=M$ or $H_f=P_1\times P_1$ and we are in Row 4. \end{proof} 

Now we determine some specific conditions for $H_f$ to contain certain subgroups of $H$. Recall the action of $H$ on $N=\{f\in T^H\mid f(z\ell)=f(z)^{\phi(\ell)}\quad\forall z\in H, \ell\in L\}$, where $H=(T\times T)\rtimes\langle\iota\rangle$, $L=\{(x,x)\mid x\in T\}\langle\iota\rangle$, and $\phi((x,x)\iota^k)=i_x$ for all $(x,x)\iota^k\in L$ ($i_x$ denotes the automorphism induced on $T$ by conjugation by $x$). 

\begin{lemma}\label{XY} Let $f\in N$ and set $\alpha(t)=f((t,1))$ for all $t\in T$. Let $X$ and $Y$ be subgroups of $T$. Then $X\times Y\leqslant H_f$ if and only if the following conditions hold:
\begin{itemize}
\item $\alpha(xt)=\alpha(t)$ for all $x\in X$, $t\in T$;
\item $\alpha(ty)=\alpha(t)^y$ for all $y\in Y$, $t\in T$.
\end{itemize}
In particular, $\alpha(x)=\alpha(1)$ for all $x\in X$. 
\end{lemma}

\begin{proof} We will use Lemma \ref{equivalent N} here. We have $(x,y)\in H_f$ for all $x\in X$, $y\in Y$ if and only if
\begin{align*} &f^{(x,y)}((a,b)\iota^k)=f((a,b)\iota^k)\quad\forall a,b\in T, \forall x\in X, \forall y\in Y, \forall k\in\{0,1\}\\
&\iff f((xa,yb)\iota^k)=f((a,b)\iota^k)\quad\forall a,b\in T, \forall x\in X, \forall y\in Y, \forall k\in\{0,1\}\\
&\iff f((xab^{-1}y^{-1},1))^{yb}=f((ab^{-1},1))^b\quad\forall a,b\in T, \forall x\in X, \forall y\in Y\\
&\iff \alpha(xty^{-1})^{y}=\alpha(t)\quad\forall t\in T, \forall x\in X, \forall y\in Y. \end{align*}
We claim that this final condition is equivalent to 
\begin{equation}\label{conditions1} \alpha(xt)=\alpha(t) \text{ and } \alpha(t)^y=\alpha(ty)\quad\forall t\in T, x\in X, y\in Y. \end{equation} 
Indeed, if $\alpha(xty^{-1})^{y}=\alpha(t)$ for all $t\in T$, $x\in X$ and $y\in Y$, we can set $y=1$ and $x=1$ respectively to obtain the equations in \eqref{conditions1}. Conversely, if the equations in \eqref{conditions1} hold then for any $t\in T$, $x\in X$ and $y\in Y$ we have
\[ \alpha(t)=\alpha(xt)=\alpha(xty^{-1}y)=\alpha(xty^{-1})^{y}. \]
So $X\times Y\leqslant H_f$ if and only if the equations in \eqref{conditions1} hold. \end{proof} 

\begin{lemma}\label{TW centraliser} Let $K$ be a subgroup of $T$. If a function $\alpha:T\rightarrow T$ satisfies $\alpha(kt)=\alpha(t)$ and $\alpha(tk)=\alpha(t)^k$ for all $k\in K$ and $t\in T$, then $K\cap K^t$ is contained in $\textbf{C}_T(\alpha(t))$ for all $t\in T$. \end{lemma}

\begin{proof} Take any $k\in K\cap K^t$. Then $tkt^{-1}\in K$, so 
\[ \alpha(t)^k=\alpha(tk)=\alpha((tkt^{-1})t)=\alpha(t). \]
Thus $K\cap K^t\leqslant\textbf{C}_T(\alpha(t))$. \end{proof}

\begin{lemma}\label{PSLcentraliser} Let $T=\text{PSL}(2,q)$ for $q\geqslant 4$. Suppose a function $\alpha:T\rightarrow T$ satisfies $\alpha(pt)=\alpha(t)$ and $\alpha(tp)=\alpha(t)^p$ for all $p\in P_1$ and $t\in T$. Then $\alpha(t)=1$ for all $t\in P_1$. \end{lemma} 

\begin{proof} By Lemma \ref{TW centraliser} we get $P_1\leqslant \textbf{C}_T(\alpha(1))$. This can only happen if $\alpha(1)=1$ and so $\alpha(t)=1$ for all $t\in P_1$. \end{proof} 

\begin{corollary}\label{TP_1} Let $T=\text{PSL}(2,q)$ for $q\geqslant 4$. Then there does not exist $f\in N$ such that $H_f=T\times P_1$. \end{corollary}

\begin{proof} Suppose for a contradiction that such an $f$ exists. Set $\alpha(t)=f((t,1))$ for all $t\in T$. Then by Lemma \ref{XY} we have $\alpha(xt)=\alpha(t)$ and $\alpha(tp)=\alpha(t)^p$ for all $x\in T$, $p\in P_1$ and $t\in T$. We can apply Lemma \ref{PSLcentraliser} to get $\alpha(1)=1$. Then since $\alpha(xt)=\alpha(t)$ for all $x\in T$ and $t\in T$, it follows that $\alpha$ is a constant function, so $\alpha(t)=1$ for all $t\in T$. This is a contradiction as then $H_f=H$. \end{proof}

\begin{lemma}\label{conditions g} Let $f\in N$ and set $\alpha(t)=f((t,1))$ for all $t\in T$. Let $K$ be a maximal subgroup of $T$. Then $H_f=K\wr S_2=(K\times K)\rtimes\langle\iota\rangle$ if and only if the following conditions hold:
\begin{itemize}
\item $\alpha(kt)=\alpha(t)$ for all $k\in K$, $t\in T$;
\item $\alpha(t)=\alpha(t^{-1})^t$ for all $t\in T$;
\item there exists $t\in T$ such that $\alpha(t)\neq 1$. 
\end{itemize}
\end{lemma} 

\begin{proof} By Lemma \ref{XY} we obtain $K\times K\leqslant H_f$ if and only if 
\begin{equation}\label{conditions2} \alpha(kt)=\alpha(t) \text{ and } \alpha(t)^k=\alpha(tk)\quad\forall t\in T, k\in K. \end{equation} 
We now determine an equivalent set of conditions for $(K\times K)\iota\subseteq H_f$ to hold. By Lemma \ref{equivalent N} we obtain $(k_1,k_2)\iota\in H_f$ for all $k_1,k_2\in K$ if and only if
\begin{align*} &f^{(k_1,k_2)\iota}((a,b)\iota^x)=f((a,b)\iota^x)\quad\forall a,b\in T, \forall k_1,k_2\in K, \forall x\in\{0,1\}\\
&\iff f((k_1b,k_2a)\iota^{x+1})=f((a,b)\iota^x)\quad\forall a,b\in T, \forall k_1,k_2\in K, \forall x\in\{0,1\}\\
&\iff f((k_1ba^{-1}k_2^{-1},1))^{k_2}=f((ab^{-1},1))^{ba^{-1}}\quad\forall a,b\in T, \forall k_1,k_2\in K\\
&\iff \alpha(k_1tk_2^{-1})^{k_2}=\alpha(t^{-1})^t\quad\forall t\in T, \forall k_1,k_2\in K. \end{align*}
We claim that this final condition is equivalent to
\begin{equation}\label{conditions3} \alpha(kt)=\alpha(t) \text{ and } \alpha(tk^{-1})^k=\alpha(t^{-1})^t\quad\forall t\in T, k\in K. \end{equation}
Indeed, if $\alpha(k_1tk_2^{-1})^{k_2}=\alpha(t^{-1})^t$ for all $t\in T$ and $k_1,k_2\in K$, we can set $k_1=1$ to obtain the second equation in \eqref{conditions3}. By setting $k_2=1$ and comparing this with the equation where $k_1=k_2=1$, we obtain the first equation in \eqref{conditions3}. Conversely, if the equations in \eqref{conditions3} hold then for any $t\in T$ and $k_1,k_2\in K$ we have
\[ \alpha(k_1tk_2^{-1})^{k_2}=\alpha(tk_2^{-1})^{k_2}=\alpha(t^{-1})^t. \]
So $(K\times K)\iota\subseteq H_f$ if and only if the equations in \eqref{conditions3} hold. 

Putting everything together we get $(K\times K)\rtimes\langle\iota\rangle\leqslant H_f$ if and only if 
\begin{equation}\label{conditions4} \alpha(kt)=\alpha(t), \quad \alpha(t)^k=\alpha(tk) \text{ and } \alpha(tk^{-1})^k=\alpha(t^{-1})^t\quad\forall t\in T, k\in K. \end{equation} 
Next, we show these conditions are equivalent to
\begin{equation}\label{conditions5} \alpha(kt)=\alpha(t)\text{ and } \alpha(t)=\alpha(t^{-1})^t\quad\forall t\in T, k\in K. \end{equation} 
Indeed, \eqref{conditions5} follows from \eqref{conditions4} as the first equation is the same in both and setting $k=1$ in the third equation of \eqref{conditions4} yields the second equation in \eqref{conditions5}. Conversely, suppose \eqref{conditions5} holds. Then
\[ \alpha(tk)=\alpha(k^{-1}t^{-1})^{tk}=\alpha(t^{-1})^{tk}=\alpha(t)^k\quad\forall t\in T, k\in K, \]
proving the second equation of \eqref{conditions4}. Also,
\[ \alpha(tk^{-1})^k=\alpha(kt^{-1})^{(tk^{-1})k}=\alpha(kt^{-1})^t=\alpha(t^{-1})^t\quad\forall t\in T, k\in K, \]
proving the third equation in \eqref{conditions4}. This proves the claim that $(K\times K)\rtimes\langle\iota\rangle\leqslant H_f$ if and only if the equations in \eqref{conditions5} hold. 

By Lemma \ref{maximal H} we have that $(K\times K)\rtimes\langle\iota\rangle$ is maximal in $H$, and so $H_f=(K\times K)\rtimes\langle\iota\rangle$ if and only if $(K\times K)\rtimes\langle\iota\rangle\leqslant H_f$ and $H_f < H$. Now $H_f<H$ if and only if there exists some $h\in H$ such that $f(h)\neq 1$, by Lemma \ref{basic}. By Lemma \ref{equivalent N}, this is equivalent to the existence of $t\in T$ with $\alpha(t)\neq 1$. Thus $H_f=(K\times K)\rtimes\langle\iota\rangle$ if and only if the equations in \eqref{conditions5} hold and there exists $t\in T$ such that $\alpha(t)\neq 1$. \end{proof} 

\begin{lemma}\label{P_1 wr S_2} Let $T=\text{PSL}(2,q)$ where $q\geqslant 4$ and either $q$ is even or $q\equiv 3\pmod{4}$. Then there does not exist $f\in N$ such that $H_f=P_1\wr S_2$. \end{lemma}

\begin{proof} Suppose for a contradiction that such an $f$ exists. Define $\alpha:T\rightarrow T$ by $\alpha(t)=f((t,1))$ for all $t\in T$. Then since $P_1\times P_1\leqslant H_f$, Lemma \ref{XY} implies that $\alpha(pt)=\alpha(t)$ and $\alpha(tp)=\alpha(t)^p$ for all $p\in P_1$ and $t\in T$. By Lemma \ref{PSLcentraliser} we conclude that $\alpha(t)=1$ for all $t\in P_1$. 

Now let $t$ be an involution in $T$ that is not contained in $P_1$. By Lemma \ref{conditions g} we have $\alpha(t)^t=\alpha(t^{-1})^t=\alpha(t)$, so $t\in\textbf{C}_T(\alpha(t))$. So Lemma \ref{TW centraliser} implies that $P_1\cap P_1^t < \langle P_1\cap P_1^t, t\rangle\leqslant \textbf{C}_T(\alpha(t))$. Now $X=\langle P_1\cap P_1^t, t\rangle \cong D_{\frac{2(q-1)}{(2,q-1)}}$, which we can see by observing that $t$ normalises $P_1\cap P_1^t\cong C_{\frac{q-1}{(2,q-1)}}$, and consulting the listed subgroups of $\text{PSL}(2,q)$ in \cite{PSLMoore}. Hence if $X\leqslant\textbf{C}_T(\alpha(t))$ then $\alpha(t)\in Z(X)=1$ as $q\not\equiv 1\pmod{4}$. By Lemma \ref{P_1 involution} it follows that $\alpha(t)=1$ for all $t\notin P_1$, as $\alpha$ is constant on the right cosets of $P_1$ in $T$. Since we previously showed that $\alpha(t)=1$ for all $t\in P_1$, we conclude that $\alpha(t)=1$ for all $t\in T$. But then $H_f=H$, a contradiction, so no such $f$ exists. \end{proof}  

\begin{lemma}\label{P_1P_1 iff} Let $T=\text{PSL}(2,q)$ where $q\geqslant 4$ and either $q$ is even or $q\equiv 3\pmod{4}$. Let $f\in N\backslash\{\text{id}\}$. Then whenever $P_1\times P_1\leqslant H_f$, we have $P_1\times P_1=H_f$. \end{lemma} 

\begin{proof} The only proper subgroups of $H$ containing $P_1\times P_1$ are $T^2, P_1\times P_1, P_1\wr S_2, T\times P_1$ and $P_1\times T$. Recall that the last three have been disproved in Lemma \ref{P_1 wr S_2} and Corollary \ref{TP_1} by noting that $T\times P_1$ and $P_1\times T$ are conjugate in $H$. Finally, we cannot have $H_f=T^2$ by Lemma \ref{no T^m}. \end{proof}

\begin{remark}\label{must have P_1P_1} We can now say more about the exact value of the subdegree in Corollary \ref{P_1P_1 not maximal}. If $q$ is even or $q\equiv 3\pmod{4}$ then we claim the subdegree equals $2(q+1)^2$. Indeed, the proof of Corollary \ref{P_1P_1 not maximal} implies there exists $f\in N\backslash\{\text{id}\}$ such that $P_1\times P_1\leqslant H_f$. So by Lemma \ref{P_1P_1 iff} we must have $P_1\times P_1=H_f$, so the subdegree is $|H:H_f|=2(q+1)^2$. \end{remark}

\begin{theorem}\label{subdegree classification} The group $G(2,q)$ has nontrivial coprime subdegrees if and only if $q\equiv 3\pmod{4}$ or $q=29$. Furthermore, $|H:H_f|$ and $|H:H_g|$ are nontrivial coprime subdegrees if and only if, up to reordering, one of the possibilities in Table \ref{final list} holds. Moreover, in each of these cases there exist $f,g\in N$ such that $H_f$ and $H_g$ are as in the table below. 

\begin{table}
\caption {Classification of nontrivial coprime subdegrees of $G(2,q)$} \label{final list}
\begin{center}
\begin{tabular}{ |c|c|c|c|c| } 
 \hline
 $q$ & $H_f$ & $H_g$ & $|H:H_f|$ & $|H:H_g|$\\ 
 \hline
$\equiv 3\pmod{4}$& $P_1\times P_1$& $D_{q+1}\wr S_2$& $2(q+1)^2$& $(\frac{1}{2}q(q-1))^2$ \\
$7$& $C_7\times C_7\leqslant H_f$& $S_4\wr S_2$& \textrm{divides} $2(24^2)$ & $7^2$ \\
$11$& $C_{11}\times C_{11}\leqslant H_f$& $A_5\wr S_2$& \textrm{divides} $2(60^2)$ & $11^2$ \\
$11$& $P_1\times P_1$& $A_4\wr S_2$& $2(12^2)$ & $55^2$ \\
$19$& $P_1\times P_1$& $A_5\wr S_2$& $2(20^2)$& $57^2$ \\
$23$& $P_1\times P_1$& $S_4\wr S_2$& $2(24^2)$& $253^2$ \\
$29$& $X\times X\leqslant H_f$ where $|P_1:X|=2$& $A_5\wr S_2$& \textrm{divides} $2(60^2)$ & $203^2$ \\
$59$& $P_1\times P_1$& $A_5\wr S_2$& $2(60^2)$& $1711^2$ \\
 \hline
\end{tabular}
\end{center}
\end{table}
\end{theorem}

\begin{proof} We refer to Table \ref{big}, and use Lemmas \ref{P_1 wr S_2} and \ref{TP_1} to eliminate all the cases where $H_f= P_1\wr S_2$ or $H_f=T\times P_1$. This leaves us with the cases in Table \ref{final list}, and it is easy to check that in each of these cases we obtain nontrivial coprime subdegrees. Thus if $G(2,q)$ has nontrivial coprime subdegrees, then $q\equiv 3\pmod{4}$ or $q=29$, and Theorem \ref{infinite families} shows that for all these values of $q$, the group $G(2,q)$ has nontrivial coprime subdegrees. So the first part of this theorem has been proven. 

To see that in each of these cases there exist $f,g\in N$ such that $H_f$ and $H_g$ are as shown in Table \ref{final list}, we refer to Remark \ref{must have P_1P_1} and Sections \ref{TW} and \ref{T=PSL}. \end{proof}

\begin{remark} Suppose $q\equiv 3\pmod{4}$. Then $G(2,q)$ has the pair of nontrivial coprime subdegrees $(2(q+1)^2, (\frac{1}{2}q(q-1))^2)$. If $q\neq\{7,11,19\}$ this is the only such pair. If $q\in\{7,11\}$ there are at least two such pairs, and if $q=19$ there are exactly two such pairs, with the other pair being $(2(20)^2, 57^2)$.  \end{remark}

\end{document}